\documentclass[a4paper]{amsart}
\usepackage{url}
\usepackage{amsmath}
\usepackage[makeroom]{cancel}

\usepackage{graphicx}
\usepackage[all]{xy}
\usepackage[mathscr]{eucal}
\usepackage{amsmath,amssymb,amsfonts}

\newtheorem{theorem}{Theorem}[section]
\newtheorem{lemma}[theorem]{Lemma}
\newtheorem{corollary}[theorem]{Corollary}
\newtheorem{example}[theorem]{Example}

\newtheorem{proposition}[theorem]{Proposition}
\newtheorem{remark}[theorem]{Remark}

\usepackage{stackengine}
\usepackage{xcolor}
\usepackage[all]{xy}

\providecommand{\keywords}[1]{\textbf{Keywords---} #1}
\usepackage{orcidlink}

\title{Some \'etale groupoids as groupoids of fractions}

\author{Mark V. Lawson\,
\orcidlink{0000-0001-8868-6796}}

\address{Mark V. Lawson, Department of Mathematics,
Maxwell Institute for Mathematical Sciences,
Heriot-Watt University,
Riccarton,
Edinburgh EH14 4AS,
UNITED KINGDOM}
\email{m.v.lawson@hw.ac.uk}

\begin{document}

\begin{abstract}
Given a monoid that can be embedded in a group and that acts partially on a set, we show that
under certain conditions the category associated with the partial action has a groupoid of fractions.
By imposing some topological conditions on the action,
we show that the groupoid is actually \'etale.
This \'etale groupoid arises naturally in $C^{\ast}$-algebra theory.
\end{abstract}


\maketitle

\section{background}

This paper arose from a simple question: 
why do the Cuntz and Cuntz-Krieger groupoids, described in, for example,  \cite[pp~153--157]{P}, have the form that they do?
There are two answers to this question.
One is that it is what it is.
The other seeks a deeper reason; this is the goal of this paper.
Part of the answer can be found in \cite{RW} in their notion of `partial monoid action'
(which we term `strongly partial action' in this paper).
But their groupoid then arises as a {\em deus ex machina}.
In this paper, we shall show that the groupoid arises naturally as a groupoid of fractions.
In other words, the groupoid has structure.
The ideas which underpin this paper can be found in \cite[pp~33,~34]{LV}.

All categories will be small and treated as algebraic generalizations of monoids;
thus for us a category is like a monoid with many identities.
If $a$ is an element of a category then $\mathbf{d}(a)$,
the {\em domain} of $a$, will be the unique identity such that $\mathbf{d}(a)a$ is defined
and $\mathbf{r}(a)$, the {\em range} of $a$, will be the unique identity such that $a\mathbf{r}(a)$ is defined.
We shall compose elements in categories from left to right: thus $\exists ab$ precisely when
$\mathbf{r}(a) = \mathbf{d}(b)$. 
The partial multiplication in a category is denoted by $\mathbf{m}$. 
We shall identify objects with identities.
If $C$ is a category then its sets of identities will be denoted by $C_{o}$.
If the category $C$ is a subcategory of the category $D$ so that $C_{o} = D_{o}$,
then we say that $C$ is a {\em wide} subcategory of $D$.
A {\em topological category} $C$ is one equipped with a topology so that the maps $\mathbf{d}$, $\mathbf{r}$ and $\mathbf{m}$
are continuous.
A {\em cancellative category} is one in which if $ab = ac$ is defined then $b = c$,
and if $ba = ca$ is defined then $b = c$.
A {\em groupoid} is a category such that for every element $g$ there is a (perforce) unique element $g^{-1}$
such that $gg^{-1} = \mathbf{d}(g)$ and $g^{-1}g = \mathbf{r}(g)$. 
If $X$ is any non-empty set and $G$ is any group, then we can turn $X \times G \times X$ into a groupoid
when we make the following definitions: $\mathbf{d}(x,g,y) = (x,1,x)$, $\mathbf{r}(x,g,y) = (y,1,y)$,
$(x,g,y)^{-1} = (y,g^{-1}, x)$ and $(x,g,y)(y,h,z) = (x,gh,z)$. 
This paper is in two parts: an algebraic (in Section 2) and a topological (in Section 3).

The ideas of this paper were applied to semigroup theory in \cite{Lawson2026}.

\section{Algebraic constructions}

We take our initial cue from \cite[Section 5]{RW}, although our notion of `partial monoid action' is weaker than theirs.

Let $M$ be a monoid with identity 1, 
let $X$ be a set,
let $X \ast M \subseteq X \times M$, 
and let $f \colon X \ast M \rightarrow X$ be a function.
We write $f(x,m) = x \cdot m$, for the sake of notation,
and if $(x,m) \in X \ast M$ we write $\exists x \cdot m$.
We now consider the set of all triples $(x,m, x \cdot m) \in X \times M \times X$
such that $\exists x \cdot m$.
We denote the set of all such triples by $C(X,M)$.
We shall assume that 
$$\exists x \cdot 1$$ 
always and that 
$$x \cdot 1 = x.$$
This means that $(x,1,x) \in C(X,M)$ for all $x \in X$.
If $(x,m,y) \in C(X,M)$ define $\mathbf{d}(x,m,y) = (x,1,x)$ and $\mathbf{r}(x,m,y) = (y,1,y)$.
Following \cite{CG2012}, but weaker than \cite{RW}, we make the following definition.
We say that $M$ acts {\em partially} on $X$ (on the right) if 
$$\exists (x \cdot a) \cdot b \text{ implies that }
\exists x \cdot (ab)$$
in which case 
$$(x \cdot a) \cdot b = x \cdot (ab).$$

\begin{lemma}\label{lem:one} $C(X,M)$ is a category under the multiplication $(x,a,y)(y,b,z) = (x,ab,z)$ iff
$M$ acts partially on $X$.
\end{lemma}
\begin{proof} Suppose first that $C(X,M)$ is a category.
If $(x \cdot a) \cdot b$ is defined
then $(x,a, x \cdot a)$ and  $(x \cdot a, b, (x \cdot a) \cdot b)$ 
are both defined.
The product of these two elements is $(x,ab, (x \cdot a) \cdot b) \in C(X,M)$.
But this has to equal $(x,ab, x \cdot (ab))$.
Thus $x \cdot (ab)$ is defined and $x \cdot (ab) = (x \cdot a) \cdot b$.
We therefore have a partial action.
To prove the converse, we assume that we have a partial action (as defined)
and show that $C(X,M)$ is a category.
This is straightforward.
\end{proof}

We shall refer to the category $C(X,M)$ as the `obvious' category associated with the partial action $(X,M)$.
It is easy to characterize the categories $(C,M)$ but we shall not do so here.

If $m \in M$ we write $\mbox{\rm dom}(m)$ for the set of all $x \in X$ such that $\exists x \cdot m$.
We denote by $\mbox{\rm im}(m)$ the set of all elements of the form $x \cdot m$ for some $x \in X$.
It follows that $m$ induces a function $f_{m} \colon \mbox{\rm dom}(m) \rightarrow \mbox{\rm im}(m)$
given by $(x)f_{m} = x \cdot m$.

\begin{remark}
{\em  In this section, we shall assume that $M$ acts partially on $X$ so that $C(X,M)$ is a category.
We have apparently replaced one complicated thing --- a partial action --- by another complicated thing --- a category.
This paper argues that there is profit in doing this.}
\end{remark}

We shall now investigate conditions on our action so that the category $C(X,M)$ is `nice'.

\begin{lemma}\label{lem:two} Let $M$ act partially on the set $X$.
If $M$ is cancellative then the category $C(X,M)$ is cancellative.
\end{lemma}
\begin{proof} Assume that $M$ is cancellative.
Suppose that $(x,a,x \cdot a)(x \cdot a, b, x \cdot ab) = (x,a,x \cdot a)(x \cdot a, c, x \cdot ac)$.
Then we deduce that $ab = ac$ and so, by left cancellation, $b = c$.
The result now follows by symmetry.
\end{proof}

In what follows, we need a stronger assumption on our action than being partial in the above sense.
We say that a partial action is {\em strongly partial} if the following assumption holds:
$$\exists (x \cdot a) \cdot b \text{ iff } \exists x \cdot (ab);$$
in which case, $x \cdot (ab) = (x \cdot a) \cdot b$.
\footnote{Our strongly partial actions are simply the partial actions of \cite{RW}.}

We say that a category $C$ is {\em left reversible} if for all elements $a,b \in C$ such that $\mathbf{d}(a) = \mathbf{d}(b)$
we have that $aC \cap bC \neq \varnothing$;
we are using the same terminology as \cite{CP} since we regard small categories as generalizations of monoids.
The following result is simply a translation into strongly partial actions of the definition of left reversible.

\begin{lemma}\label{lem:ten} Let the monoid $M$ act strongly partially on the set $X$.
Then the category $C(X,M)$ is left reversible iff whenever $x \in \mbox{dom}(m) \cap \mbox{dom}(n)$
there exist $a,b \in M$ such that $ma = nb$ where $x \cdot m \in \mbox{dom}(a)$ and $x \cdot n \in \mbox{dom}(b)$.
\end{lemma}
\begin{proof} Suppose first that the following condition holds:
whenever $x \in \mbox{dom}(m) \cap \mbox{dom}(n)$
there exist $a,b \in M$ such that $ma = nb$ where $x \cdot m \in \mbox{dom}(a)$ and $x \cdot n \in \mbox{dom}(b)$.
We prove that $C(X,M)$ is left reversible.
Consider the following elements 
$(x,m,x \cdot m)$
and
$(x,n,x \cdot n)$
which are arbitrary apart from having the same domain.
Observe that $x \in \mbox{dom}(m) \cap \mbox{dom}(n)$.
Thus there are elements $m$ and $n$ such that 
$ma = nb$ where $x \cdot m \in \mbox{dom}(a)$ and $x \cdot n \in \mbox{dom}(b)$.
We may therefore define 
$(x \cdot m, a, (x \cdot m) \cdot a)$ and $(x \cdot n,b,(x \cdot n) \cdot b)$,
both elements of $C(X,M)$, 
satisfying
$$(x,m, x \cdot m)(x \cdot m, a, (x \cdot m) \cdot a) = (x,n,x\cdot n)(x \cdot n,b,(x \cdot n) \cdot b).$$
This proves that $C(X,M)$ is right reversible.
To prove the converse, suppose that $C(X,M)$ is right reversible.
Let $x \in \mbox{dom}(m) \cap \mbox{dom}(n)$.
Then $(x,m,x \cdot m)$ and $(x,n,x \cdot n)$ are well-defined elements of $C(X,M)$ with the same domain.
By assumption, 
$$(x,m,x \cdot m)C(X,M) \cap (x,n,x \cdot n)C(X,M) \neq \varnothing.$$
Thus 
$$(x,m, x \cdot m)(x \cdot m, a, (x \cdot m) \cdot a) = (x,n,x\cdot n)(x \cdot n,b,(x \cdot n) \cdot b)$$
for some elements $(x \cdot m, a, (x \cdot m) \cdot a)$ and $(x \cdot n,b,(x \cdot n) \cdot b)$ of $C(X,M)$.  
It follows that $ma = nb$ and $\exists (x \cdot m) \cdot a$ and $\exists (x \cdot n) \cdot b$.
\end{proof}

We shall now define a class of strongly partial actions of a monoid $M$ on a set $X$ (following \cite{RW}):
we shall say that such an action is {\em directed} if, whenever $\mbox{\rm dom}(m) \cap \mbox{\rm dom}(n) \neq \varnothing$,
there are elements $a,b \in M$ such that $ma = nb = r$  (say) and $\mbox{\rm dom}(m) \cap \mbox{\rm dom}(n) = \mbox{\rm dom}(r)$.
It is easy to check that directed strongly partial actions satisfy the condition of Lemma~\ref{lem:ten}.
The proof of the following is now immediate.

\begin{proposition}\label{prop:three} Let there be a strongly partial action of the monoid $M$ on the set $X$.
If this action is directed then the category $C(X,M)$ is left reversible
\end{proposition}

\begin{remark}\label{rem:trump} {\em If the action is globally defined then $M$ left reversible implies that $C(X,M)$ is left reversible.
This accords with \cite[Example 5.3]{RW}.}
\end{remark}

In what follows, we shall assume that $M$ is a cancellative monoid
and so by Lemma~\ref{lem:two} the category $C(X,M)$ is cancellative.
We shall also assume that 
we have a directed strongly partial action.
Under these assumptions and by Proposition~\ref{prop:three}, we know that
$C(X,M)$ is a left reversible cancellative category.
However, any left reversible cancellative category can be embedded in a groupoid in a particularly nice way.
The following is well-known, being part of \cite{GZ}, but we provide a proof anyway because of its importance.

\begin{theorem}\label{them:four}
Let $C$ be a left reversible cancellative category.
Then there is a groupoid $G$ such that
$C$ is embedded in $G$ by a functor $\iota$ such that
$\iota (C)$ is a wide subcategory of $G$
and 
$\iota (C)\iota (C)^{-1} = G$.
The groupoid $G$ is unique with these properties.
\end{theorem}
\begin{proof} We shall build $G$ from $C$.
Consider the ordered pair $(x_{1},y_{1})$ where $\mathbf{r}(x_{1}) = \mathbf{r}(y_{1})$.
We would like to regard this as the element $x_{1}y_{1}^{-1}$ in the groupoid $G$.
If $(x_{2},y_{2})$ is such that $\mathbf{d}(x_{1}) = \mathbf{d}(x_{2})$ and $\mathbf{d}(y_{1}) = \mathbf{d}(y_{2})$.
Define $(x_{1},y_{1}) \sim (x_{2},y_{2})$ iff there exist $a,b \in C$ such that $(x_{1},y_{1})a = (x_{2},y_{2})b$.
using the fact that $C$ is left reversible, it can be shown that $\sim$ is an equivalence relation.
We denote the $\sim$-equivalence class containing $(x,y)$ by $[x,y]$.
Define $\mathbf{r}[x,y] = [\mathbf{d}(y), \mathbf{d}(y)]$ and $\mathbf{d}[x,y] = [\mathbf{d}(x), \mathbf{d}(x)]$.
Define $[x,y]^{-1} = [y,x]$.
To define the product $[x,y][u,v]$ we require that $\mathbf{d}(y) = \mathbf{d}(u)$.
By left reversibility, we have that there exist elements $a,b \in C$ such that $ya = ub$.
Define $[x,y][u,v] = [xa,vb]$.
Judicious use of left reversibility shows that this is well-defined.
Put $G$ equal to the set of $\sim$-equivalence classes, together with the above operations.
It is now easy to show that $G$ is a groupoid.
Define $\iota \colon C \rightarrow G$
by $\iota (x) = [x,\mathbf{r}(x)]$.
It is routine to check that $\iota$ is a functor.
It is an embedding because $C$ is cancellative.
The element $[x,y] = [x,\mathbf{r}(x)][y,\mathbf{r}(y)]^{-1}$ and so 
$\iota (C)\iota (C)^{-1} = G$.
The identities of $G$ are the elements of the form $[e,e]$ where $e$ is an identity of $C$.
Thus $\iota (C)$ is a wide subcategory of $G$.

Now, let $H$ be a groupoid equipped with an embedding $\theta \colon C \rightarrow H$ 
such that the properties stated in the theorem all hold.
We shall define an isomorphism $\phi \colon G \rightarrow H$.
Let $g \in G$.
By assumption, we can write $g = \iota (x)\iota (y)^{-1}$ where $x,y \in C$.
But, if $g = \iota (x') \iota (y')^{-1}$,
then $(x,y) \sim (x',y')$.
It follows that $(x,y) \sim (x',y')$.
We may therefore define $\phi (g) = \theta (x)\theta (y)^{-1}$.
We show that $\phi$ is injective.
Suppose that $g = \iota (x)\iota (y)^{-1}$ and $h = \iota (u)\iota (v)^{-1}$ and that $\phi (g) = \phi (h)$.
You can quickly deduce that $\mathbf{r}(x) = \mathbf{r}(y)$, that $\mathbf{r}(u) = \mathbf{r}(v)$
and that $\mathbf{d}(x) = \mathbf{d}(u)$ and that $\mathbf{d}(y) = \mathbf{d}(v)$.
Let $a,b \in C$ be such that $xa = ub$.
We prove that $ya = vb$ from which injectivity follows.
However, it is easy to show that $\theta (ya) = \theta (vb)$.
But $\theta$ is injective.
Surjectivity follows from the fact that $H = \theta (C)\theta (C)^{-1}$.
It is now routine to check that $\phi$ is a functor. 
\end{proof}

The groupoid $G$ is called a {\em groupoid of fractions} of $C$.

We now return to our category $C(X,M)$.
To simplify calculations, we shall assume that $M$ can be embedded in a group $N$.
We describe the groupoid of fractions of $C(X,M)$, denoted by $G(X,M)$, in another way.
We describe another groupoid, denoted by $\mathscr{G}$, which was defined in \cite[Lemma 5.5]{RW}.
It consists of triples $(x,q,y) \in X \times N \times X$ such that $q = mn^{-1}$ where $m,n \in M$
and $x \cdot m = y \cdot n$.
You should regard $\mathscr{G}$ as a subgroupoid of $X \times N \times X$.
We have the following result which connects our approach with the one in \cite{RW}.

\begin{theorem}\label{them:five} Let $M$ have a directed strongly partial action on the set $X$,
and suppose that $M$ can be embedded in a group $N$.
Then the category $C(X,M)$ has a groupoid of fractions $G(X,M)$ which is
isomorphic with the groupoid $\mathscr{G}$ defined above. 
\end{theorem}
\begin{proof} We could define an explicit isomorphism, but we instead will use the uniqueness guaranteed by
Theorem~\ref{them:four}.
Observe that we can regard $C(X,N)$ as a subcategory of $\mathscr{G}$
since $m = m1^{-1}$ for each element of $m$.
The identities of $\mathscr{G}$ have the form $(x,1,x)$.
It follows that $C(X,M)$ is a wide subcategory of $\mathscr{G}$.
Finally, if $(x,q,y)$ is an element of $\mathscr{G}$ where $q = mn^{-1}$ and $x \cdot m = y \cdot n$
then 
$$(x,m,x \cdot m)(y,n,x \cdot m)^{-1},$$
where $(x,m,x \cdot m),(y,n,x \cdot m) \in \mathscr{G}$.
\end{proof}

\begin{example}\label{ex:mark}
{\em We now return to the groupoid that motivated this paper.
Let $G$ be a directed graph.
We denote the right-infinite paths in $G$ by $G^{\omega}$ and the finite paths in $G$ by $G^{\ast}$.
If $x \in G^{\omega}$ then $x = x_{0}x_{1}x_{2} \ldots$
where the $x_{i}$ are edges of $G$.
The monoid $\mathbb{N}$ acts on the right on the set $G^{\omega}$
by $x \cdot n = x_{n}x_{n+1}x_{n+2} \ldots$.
Thus, the effect of $n$ is to chop off the finite string at the front of $x$ of length $n$;
this is nothing more than the iterated version of the shift map.
There is an `obvious' category associated with this action: $C(G^{\omega}, \mathbb{N})$.
The monoid $\mathbb{N}$ can be embedded in the group $\mathbb{Z}$.
Because the action is global, and $\mathbb{N}$ is left reversible, it follows that $C(G^{\omega}, \mathbb{N})$ is left reversible by Remark~\ref{rem:trump} 
(although this can easily be seen directly).
We use Theorem~\ref{them:five} to describe the groupoid of fractions of the category $C(G^{\omega}, \mathbb{N})$.
It consists of triples $(x,q,y)$, where $x,y \in G^{\omega}$, such that if $q = m - n$, where $m,n \in \mathbb{N}$, then $x \cdot m = y \cdot n$.
Thus there is a finite path $u$ of length $m$ and a finite path $v$ of length $n$ and  an infinite path $w$  such that
$x = uw$ and $y = vw$.
Thus the elements of $\mathcal{G}$ can be written $(uw, |u| - |v| , vw) \in G^{\omega} \times \mathbb{Z} \times G^{\omega}$
where $u,v \in G^{\ast}$ and $w \in G^{\omega}$.
The groupoid operations are those inherited from the groupoid $G^{\omega} \times \mathbb{Z} \times G^{\omega}$.
}
\end{example}

We may summarize what we have described in this section, as follows.
Let the monoid $M$ act partially on the set $X$.
Then, by Lemma~\ref{lem:one}, there is an associated `obvious' category $C(X,M)$.
If the monoid $M$ is actually cancellative,
then by Lemma~\ref{lem:two} $C(X,M)$ is cancellative.
By Proposition~\ref{prop:three},
if the partial action of $M$ on $X$ is strong and directed
then $C(X,M)$ is left reversible.
Therefore, if $M$ is a cancellative monoid with a directed strongly partial action on the set $X$,
then by Theorem~\ref{them:four}, the category $C(X,M)$ is a left reversible cancellative category and so has a groupoid of fractions.
If, in addition, the cancellative monoid $M$ can be embedded in a group
then, by Theorem~\ref{them:five}, we obtain a much simpler description of this groupoid of fractions.
It is this final version of the groupoid of fractions which makes contact with \cite[Lemma 5.5]{RW}.

This completes the purely algebraic part of this paper.
We now turn to the topological part.

\section{Topological constructions}

We now want to endow our objects with a topology.

A topological category is said to be {\em $\mathbf{d}$-\'etale} if $\mathbf{d}$ is a local homeomorphism;\footnote{Cockett and Garner \cite{CG} use the term `source \'etale'.
The term we use is more natural in our context.}
{\em $\mathbf{r}$-\'etale} topological categories are defined in a similar way.
A topological category which is both $\mathbf{d}$-\'etale and $\mathbf{r}$-\'etale is said to be {\em \'etale}.
Such categories were defined in \cite{KL}.
The fact that the definition in \cite{KL} agrees with the one given above was proved in \cite[Proposition 2.4]{Bice}.

We shall be interested in \'etale left reversible cancellative topological categories;
the basic question is whether the topology is inherited by the groupoid of fractions (and how).

A subset $A \subseteq C$ of a category $C$ will be called {\em $\mathbf{d}$-injective} if $a,b \in A$ and $\mathbf{d}(a) = \mathbf{d}(b)$ implies that $a = b$.
We may similarly define {\em $\mathbf{r}$-injective} subsets.
A subset which is both $\mathbf{d}$-injective and $\mathbf{r}$-injective is called a {\em local bisection}.
In a $\mathbf{d}$-\'etale category if $A$ is open then $\mathbf{d}(A)$ is open.
In such a category, we denote by $W_{x}$ an open subset that contains $x$ such that $\mathbf{d}$
restricted to $W_{x}$
induces a homeomorphism from $W_{x}$ onto $\mathbf{d}(W_{x})$.
Observe that $W_{x}$ is then a $\mathbf{d}$-injective subset.

\begin{lemma}\label{lem:base-base} Let $C$ be a topological category.
If $C$ is $\mathbf{d}$-\'etale then the open $\mathbf{d}$-injective sets form a base for the topology.
\end{lemma}
\begin{proof}
Let $x \in U$, where $U$ is any (non-empty) open subset of $C$.
Since the category is $\mathbf{d}$-\'etale, there is an open $\mathbf{d}$-injective subset $W_{x}$ of $C$ containing $x$.
Thus $x \in U \cap W_{x} = A$ is an open $\mathbf{d}$-injective subset containing $x$ but contained in $U$.
It follows that every open set is a union of $\mathbf{d}$-injective open sets.
\end{proof}

The dual of the result above also holds.

The following result does not mention the topology.

\begin{lemma}\label{lem:local-sections} In any category, the product of $\mathbf{d}$-injective subsets is a $\mathbf{d}$-injective subset.
\end{lemma}
\begin{proof} Let $A$ and $B$ be $\mathbf{d}$-injective subsets of the category $C$.
We prove that $AB$ is $\mathbf{d}$-injective.
Let $ab,cd \in AB$, where $a,c \in A$ and $b,d \in B$, such that $\mathbf{d}(ab) = \mathbf{d}(cd)$.
Then, because we are working in a category, $\mathbf{d}(ab) = \mathbf{d}(b)$ and $\mathbf{d}(cd) = \mathbf{d}(d)$.
It follows that $b = c$ since $B$ is $\mathbf{d}$-injective.
It now follows that $\mathbf{d}(a) = \mathbf{d}(c)$ and so $a = c$ since $A$ is $\mathbf{d}$-injective.
We have therefore proved that $ab = cd$.
Thus $AB$ is $\mathbf{d}$-injective subset.
\end{proof}

The following was first proved in \cite{KL}
and is also a consequence of  \cite[Proposition 2.4]{Bice}.

\begin{proposition}\label{prop:open-multiplication}
In an \'etale topological category, the multiplication map is open.
\end{proposition}

The following corollary summarizes what we need below
and uses 
Lemma~\ref{lem:base-base} (and its dual),
Lemma~\ref{lem:local-sections},
and 
Proposition~\ref{prop:open-multiplication}
in the proof of (1) and (2), whereas the proof of (3) is straightforward.

\begin{corollary}\label{cor:mary} In an \'etale topological category,
the following hold:
\begin{enumerate}
\item The product of open local bisections is an open local bisection.
\item The set of all open local bisections forms a base for the topology on $C$.
\item The set of identities is open.
\end{enumerate}
\end{corollary}

From now on, we shall deal only with an \'etale topological cancellative category $C$.
Suppose that $C$ is left reversible with groupoid of fractions $G$.
We shall prove that $G$ is an \'etale groupoid in a natural way.
Let $\beta$ be the set of all open local bisections of $C$.

\begin{lemma}\label{lem:kingdoms} 
Let $C$ be an \'etale topological cancellative left reversible category with groupoid of fractions $G$,
where $\beta$ is the set of all open local bisections of $C$.
Then $\beta \beta^{-1}$ is a base for a topology on $G$.
\end{lemma}
\begin{proof} We shall use the fact that a base for the topology in $C$ consists of open local bisections;
this is part (2) of Corollary~\ref{cor:mary}.
An element $g$ of $G$ can be written $ab^{-1}$ where $a,b \in C$.
By assumption, $a \in U \in \beta$ and $b \in V \in \beta$ for some $U$ and $V$.
It follows that $g \in UV^{-1}$.
Thus, every element of $G$ is in some element of $\beta \beta^{-1}$.

Suppose that $g \in U_{1}V_{1}^{-1} \cap U_{2}V_{2}^{-1}$, where $U_{1}, U_{2}, V_{1}, V_{2} \in \beta$.
Then we may write $g = a_{1}b_{1}^{-1} = a_{2}b_{2}^{-1}$ where $a_{1} \in U_{1}$, $a_{2} \in U_{2}$, $b_{1} \in V_{1}$
and $b_{2} \in V_{2}$.
Observe that $\mathbf{d}(a_{1}) = \mathbf{d}(a_{2})$ and $\mathbf{d}(b_{1}) = \mathbf{d}(b_{2})$.
Since $C$ is left reversible, we may find elements $a,a' \in C$ such that $a_{1}a = a_{2}a' = r$ (say)
and $b_{1}a = b_{2}a' = s$ (say).
(That the elements $a,a'$ are the same in each case follows from the two different ways that $g$ can be written).
Choose $A,B \in \beta$ such that $A$ contains $a$ and $B$ contains $a'$.
We now apply part (1) of Corollary~\ref{cor:mary}:
$U_{1}A \cap U_{2}B = X$ is an open local bisection that contains $r$
and
$V_{1}A \cap V_{2}B = Y$ is an open local bisection that contains $s$.
Thus $X,Y \in \beta$.
Observe that $rs^{-1} \in XY^{-1}$,
which simplifies to $g$.
We have shown that $g \in XY^{-1}$.
Now, let $h \in XY^{-1}$ be any element.
Then we can write $h = (\mathbf{u}\mathbf{a})(\mathbf{v}\mathbf{a}_{1})^{-1}$
where $\mathbf{u} \in U_{1}$, $\mathbf{a} \in A$, $\mathbf{v} \in V_{1}$ and $\mathbf{a}_{1} \in A$.
It follows that $h = \mathbf{u}\mathbf{a}\mathbf{a}_{1}^{-1}\mathbf{v}^{-1}$.
We must have that $\mathbf{r}(\mathbf{a}) = \mathbf{r}(\mathbf{a}_{1})$.
But $A$ is a local bisection containing both $\mathbf{a}$ and $\mathbf{a}_{1}$.
So, $\mathbf{a} = \mathbf{a}_{1}$.
We deduce that $h \in U_{1}V_{1}^{-1}$.
We may similarly show that $h \in U_{2}V_{2}^{-1}$.
We have therefore proved that $g \in XY^{-1} \subseteq U_{1}V_{1}^{-1} \cap U_{2}V_{2}^{-1}$.
\end{proof}

We now come to our main result.

\begin{theorem}\label{them:main} Let $C$ be an \'etale topological cancellative left reversible category whose groupoid of fractions $G$. 
Let $\beta$ be a base for the topology on $C$ consisting of all open local bisections 
Then $G$ is an \'etale groupoid with respect to the topology with base $\beta \beta^{-1}$.
\end{theorem}
\begin{proof} We proved in Lemma~\ref{lem:kingdoms}, that $\beta \beta^{-1}$ is a base for a topology on $G$. 
It is clear that inversion is a homeomorphism.
We show now that the partial multiplication in $G$ is continuous.
Let $g_{1}h_{1} \in UV^{-1}$ be any element.
Let $g_{1} = a_{1}b_{1}^{-1}$ and $h_{1} = a_{2}b_{2}^{-1}$ where $a_{1},b_{1},a_{2},b_{2} \in C$.
Let $p$ and $q$ be such that $b_{1}p = a_{2}q = r$.
Then $g_{1}h_{1} = (ap_{1})(b_{2}q)^{-1}$.
By assumption $g_{1}h_{1} = uv^{-1}$ where $u \in U$ and $v \in V$.
Thus $(a_{1}p)(b_{2}q)^{-1} = uv^{-1}$.
It follows that there are elements $a,b \in C$ such that
$a_{1}pa = ub$ and $b_{2}qa = vb$.
Let $Z$ and $B$ be open local bisections of $C$ such that 
$b_{1}pa \in Z$ and $b \in B$.
Consider, first, the set $UBZ^{-1}$.
We have that $a_{1}pa = ub$ and so $a_{1}pa \in UB$.
By definition, $b_{1}pa \in Z$.
It follows that $a_{1}b_{1}^{-1} \in UBZ^{-1}$
and so $g_{1} \in UBZ^{-1}$.
Now, we consider, $Z(VB)^{-1}$.
By definition, $b_{1}pa \in Z$.
But $b_{1}p = a_{2}q$.
Thus $a_{2}qa \in Z$.
On the other hand, $b_{2}qa \in VB$.
Thus $a_{2}b_{2}^{-1} \in Z(VB)^{-1}$ and so $h_{1} \in Z(VB)^{-1}$.
Now, we calculate the product
$$(UBZ^{-1})(Z(VB)^{-1}) \subseteq UBB^{-1}V^{-1} \subseteq UV^{-1}$$
where we have used the fact that $Z$ and $B$ are local bisections.
This is enough to show that the partial multiplication in the groupoid $G$ is continuous. 
We are therefore dealin with a topological groupoid.
We prove that each element of $\beta \beta^{-1}$ is a local bisection.
Let $g,h \in UV^{-1}$ such that $\mathbf{d}(g) = \mathbf{d}(h)$.
We may write $g = u_{1}v_{1}^{-1}$ and $h = u_{2}v_{2}^{-1}$.
Thus $\mathbf{d}(u_{1}) = \mathbf{d}(u_{2})$.
It follows that $u_{1} = u_{2}$.
From this we deduce that $\mathbf{d}(v_{1}) = \mathbf{d}(v_{2})$ and so $v_{1} = v_{2}$.
It follows that $g = h$.
Symmetry delivers the result.
The topology on $G$ has as a base local bisections.
From this, it follows that $G$ will be \'etale.
\end{proof}

\begin{remark}{\em By part (3) of Corollary~\ref{cor:mary}, the set of identities $C_{o}$ is an open set
in $C$. It is therefore an open local bisection.
Thus $\beta \subseteq \beta \beta^{-1}$.
This means that $C$ is an open subset of $G$, its groupoid of fractions.}
\end{remark}

We apply Theorem~\ref{them:main} to our directed strongly partial actions but add in a topology;
we do this, in essence, following \cite{RW}.
We shall assume that $X$ is a topological space (for our purposes we do not need more),
that $\mbox{dom}(m)$ and $\mbox{im}(m)$ are open subsets of $X$ for each $m \in M$,
and that $f_{m}$ (recall that $(x)f_{m} = x \cdot m$) is a local homeomorphism from $\mbox{dom}(m)$ to $\mbox{im}(m)$.
If these conditions are satisfied, then we shall refer to a {\em topological strongly partial action}.
We shall assume that we have such a partial action from now on.

If $U$ is an open subset of $X$ and $m \in M$ define
$$(U)m^{-1} = (\mbox{im}(m) \cap U)f_{m}^{-1};$$
this is an open set because $f_{m}$ is a local homeomorphism from $\mbox{dom}(m)$ to $\mbox{im}(m)$.

Let $U$ be any open subset of $X$ and $m \in M$;
define $O(U,m) = \{(x,m,x \cdot m) \colon x \in U\}$.
If $m \neq n$ then $O(U,m) \cap O(V,n) = \varnothing$
whereas
$O(U,m) \cap O(V,m) = O(U \cap V,m)$.
Let $\alpha$ be the set of all $O(U,m)$.
This is the base for a topology on $C(X,M)$.

\begin{proposition}\label{prop:eleven} If $(X,M)$ is a topological directed strongly partial action 
then $C(X,M)$ is an \'etale cancellative right reversible category with respect to the topology with base $\alpha$.
\end{proposition}
\begin{proof} We have only to show that $C(X,M)$ is an \'etale topological category.
Observe that 
$$\mathbf{d}^{-1}(O(U,1)) = \bigcup_{m \in M}O(U,m)$$
and 
$$\mathbf{r}^{-1}(O(U,1)) = \bigcup_{m \in M}O((U)m^{-1},m).$$
It follows that both $\mathbf{d}$ and $\mathbf{r}$ are continuous.
We have that $O(U,m)O((U)m^{-1},n) \subseteq O(mn,U)$.
Thus the partial multiplication in the category is continuous.
This proves that $C(X,M)$ is a topological category.
Suppose that $(x,m,x \cdot m)$ is an arbitrary  element of $C(X,M)$.
Thus, $x \in \mbox{dom}(m)$.
Let $x \in W_{x}$ be an open set on such that the restriction of $f_{m}$ to this set is a homeomorphism.
Then $\mathbf{d}$ restricted to $O(W_{x},m)$ is a homeomorphism
and $(x,m,x \cdot x) \in O(W_{x},m)$
This shows that $\mathbf{d}$ is a local homeomorphism.
We may similarly show that $\mathbf{r}$ is a local homeomorphism.
\end{proof}

The topology on $C(X,M)$ has as a base sets of the form $O(U,m)$.
In its groupoid of fractions, the set
$O(U,m)O(V,n)^{-1}$
is non-empty precisely when it has the form $Z(U,m,n,V)$
as defined prior to \cite[Lemma 5.9]{RW}.
It follows that the topology we have defined on the groupoid of fractions $G(X,M)$
is the same as the topology used in \cite[Proposition 5.12]{RW}.
Therefore by combining Proposition~\ref{prop:eleven} with Theorem~\ref{them:main},
we obtain \cite[Proposition 5.12(a)]{RW} together with the extra information that the groupoid 
in question is the groupoid of fractions of $C(X,M)$ by Theorem~\ref{them:main}. 


We now examine the converse of our main results.

\begin{theorem}\label{them:twelve} Let $G$ be an \'etale topological groupoid.
Suppose that $C$ is an open, wide subcategory of $G$ which is left reversible 
and such that $CC^{-1} = G$.
Then $C$ is an \'etale topological cancellative category,
and if $\alpha$ is a base for the topology on  $C$ consisting of open local bisections
then $\alpha \alpha^{-1}$ is a base for the topology on $G$.
\end{theorem}
\begin{proof} It is immediate that $C$ is cancellative and, because $C$ is open in $G$,
it inherits the main features of the topology of $G$.
Let $X$ be an open local bisection of $G$ containing the element $g$.
Let $g = ab^{-1}$ where $a,b \in C$.
Let $B$ be any open local bisection of $C$ that contains $b$.
Then $a \in XB$.
Thus $XB \cap C$ is non-empty.
Put $A = XB \cap C$, an open local bisection of $C$.
Observe that $a \in A$.
Thus $g \in AB^{-1}$.
We prove that $AB^{-1} \subseteq X$.
Let $h \in AB^{-1}$.
Then $h = (xb_{1})b_{2}^{-1}$, where $b_{1},b_{2} \in B$ and $x \in X$.
But $B$ is a local bisection and so $b_{1} = b_{2}$.
It follows that $h \in X$.
We have proved that each open local bisection of $G$ is a union of open local
bisections of the form $UV^{-1}$ where $U$ and $V$ are open local bisections of $C$.
\end{proof}

\begin{example}{\em  We return to the set-up of Example~\ref{ex:mark}.
A base for a topology on $C(G^{\ast}, \mathbb{N})$ has the form $O(U,m)$ where $U$ is an open subset of $G^{\omega}$.
Thus the groupoid of fractions has as a base those subsets of the form $O(U,m)O(V,n)^{-1}$.
This set consists of elements of the form $(x,m-n,y)$ where $x \cdot m = y \cdot n$ and $x \in U$ and $y \in V$.
This is precisely the set of elements of the form $(uw,|u|-|v|, vw)$ where $uw \in U$ and $vw \in V$.
If $X \subseteq G^{\omega}$ and $u$ is any finite path define $u^{-1}U = \{x \in G^{\omega} \colon ux \in U \}$.
Thus $uw \in U$ and $vw \in V$ iff $w \in u^{-1}U \cap v^{-1}V = Y$, say.
Thus we are dealing with precisely those element $(uw,|u|-|v|, vw)$ where $w \in Y$.
If $X$ is open then $u^{-1}X$ is open.
Thus $Y$ is open.
On the other hand, the set of elements of the form $(uw,|u|-|v|, vw)$ where $w \in Y'$,
where $Y'$ is open can also be written in the form
$(uw,|u|-|v|, vw)$ where $uw \in uY'$ and $vw \in vY'$.
However, if $X$ is open then $uX$ is open.
We have therefore showed that `our' topology is exactly the same as the topology in \cite[pp~153--157]{P}.}
\end{example}


\end{document}